\documentclass[10pt,a4paper]{article}
\usepackage[utf8x]{inputenc}
\usepackage{amsmath,amsthm,color}
\usepackage{amsfonts}
\usepackage{amssymb}
\usepackage{graphicx}
  \usepackage{tikz}
  \usetikzlibrary{shapes,backgrounds}

  \def\firstcircle{(0,0.5) ellipse (0.5cm and 2 cm)}
  \def\secondcircle{(-1,-1.2) ellipse (1.1cm and 0.2 cm)}
  \def\thirdcircle{(1,-1.2) ellipse (1.1cm and 0.2 cm)}
  \def\fourthcircle{(300:0 cm) circle (2cm)}

\newtheorem{theorem}{Theorem}[section]

\newtheorem{definition}[theorem]{Definition}
\newtheorem{proposition}[theorem]{Proposition}

\newtheorem{conjecture}[theorem]{Conjecture}
\newtheorem{remark}[theorem]{Remark}
\newtheorem{problem}[theorem]{Problem}
\newtheorem{example}[theorem]{Example}

\newtheorem*{ack}{Acknowledgement}

\newtheorem{lemma}[theorem]{Lemma}

\begin{document}
\title{On the structure of monomial complete intersections in positive characteristic}
\author{Samuel Lundqvist and Lisa Nicklasson}

\maketitle

\begin{abstract}
In this paper we study the Lefschetz properties of monomial complete intersections in positive characteristic. 
We give a complete classification of the strong Lefschetz property when the number of variables is at least three, which proves a conjecture by Cook II. We also extend earlier results on the weak Lefschetz property by dropping the assumption on the residue field being infinite, and by giving new sufficient criteria. 

\end{abstract}

\section{Introduction}

Let $k$ be a field and let $A=k[x_1, \ldots, x_n]/(x_1^{d_1}, \ldots, x_n^{d_n})$. The algebra $A$ is considered a graded algebra $A=\oplus_{i \ge 0} A_i$ in the usual sense, i.e. $A_i$ consists of the homogeneous polynomials of degree $i$. Recall that an artinian graded algebra has the strong Lefschetz property (SLP) if there is a linear form such that multiplication by a $d$'th power of this form has maximal rank in every degree, for all $d$. 
When the characteristic of $k$ is zero, Stanley \cite{stanley} observed that the algebra $A$ possesses the SLP. An immediate corollary to Stanley's result was a proof of the  Fr\"oberg conjecture \cite{froberg} in $n$ variables and $n +1$ forms.

When the characteristic of $k$ is positive, the algebra $A$ does not necessarily have the SLP. In fact, in many situations it also fails the weak Lefschetz property (WLP). An artinian graded algebra has the WLP if there is a linear form such that multiplication by this form has maximal rank in every degree. Thus a natural problem in characteristic $p$ is to characterize which monomial complete intersections that have the SLP and the WLP. Partial results have appeared in \cite{brenner, cook, kustin, li, lind, vraciu}. In this paper, we continue this journey using purely algebraic methods. 
For a survey of the Lefschetz properties, see \cite{atour}.

Our main result is Theorem \ref{thm:slp}, where we fully classify the SLP when $n \geq 3$. Namely, let $A = k[x_1, \ldots, x_n]/(x_1^{d_1}, \ldots, x_n^{d_n})$ where $n \ge 3$, $d_i \ge 2$ for all $i$, and where $k$ is a field of characteristic $p>0$. Let $d_1=\max(d_1, \dots, d_n)$ and write $d_1=N_1p+r_1$ with $0<r_1\le p$.
We show that $A$ has the SLP if and only if one of the following two conditions hold.
\begin{enumerate}
\item  $\sum_{i = 1}^n (d_i-1) <p$,
\item $d_1 > p$, $d_i \leq p$ for $i =2, \ldots, n$ and $\sum_{i=2}^n(d_i-1) \le \min(r_1,p-r_1)$.
\end{enumerate} 
This settles \cite[Conjecture 7.6]{cook}.

We also have results on the WLP. In Proposition \ref{prop:wlpfinite}, we remark that when $I \subset k[x_1,\ldots,x_n]$ is a monomial ideal and $k'$ is a field extension of $k$, then 
$A = k[x_1,\ldots,x_n]/I$ has the WLP if and only if $A \otimes_k k'$ has the WLP. 
This proposition makes it possible for us to finalize the classification of the WLP for uniform degrees.  

We then give a series of sufficient conditions for the presence of the WLP for mixed degrees, which we believe cover a large part of the algebras with the WLP. However, the complete characterization of the WLP is still an open question. We end up by connecting the results on the WLP to the Fr\"oberg conjecture \cite{froberg}.

 \section{Stanley's result in positive characteristic}\label{section Stanley}
We begin by adopting the proof in \cite{reid} of Stanley's result to positive characteristic. The result will later be used to give sufficient conditions for the presence of the SLP and the WLP. 

Let $A=k[x_1, \ldots, x_n]/(x_1^{d_1}, \ldots, x_n^{d_n})=\oplus_{i \ge 0} A_i$, where $k$ is a field. The \emph{Hilbert function} is defined as $H(i)=\dim_k(A_i)$, the dimension of $A_i$ as a vector space over $k$. A homogeneous element $f$ of degree $d$ defines a linear map $A_{i} \to A_{i+d}$ by $a \mapsto f \cdot a$. This map is said to have \emph{maximal rank} if it is injective or surjective. We say that multiplication by $f$ has \emph{maximal rank in every degree} if the induced multiplication map has maximal rank for each $i$.   

Let $s=x_1+ \dots + x_n$. It was proved by Stanley in \cite{stanley} that multiplication by $s^m$, where $m$ is a positive integer, has maximal rank in every degree, when the coefficient field $k$ has characteristic zero. We want to find out when this is true for a field of positive characteristic. 

Note that $t=\sum_{i=1}^n(d_i-1)$ is the greatest index for which $A_t \neq 0$. In $A_t$ we have only one power product, namely $x_1^{d_1\!-1} \cdots x_n^{d_n\!-1}$. This monomial induces a bijection $A_i \leftrightarrow A_{t-i}$, for $0 \le i \le t$, by $f \mapsto x_1^{d_1\!-1} \cdots x_n^{d_n\!-1}/f$ for monomial $f$. This shows that the Hilbert function is symmetric about $t/2$. It can also be seen that it is weakly increasing up to $t/2$. For multiplication by $s^m$ to have maximal rank in every degree, we want it to be injective up to a certain degree, and surjective for higher degree, by the symmetry of the Hilbert function. We will see later in this section that it is in fact enough to prove the injectiveness.  
To prove the injectiveness we use the same arguments as the proof of Theorem 5 in \cite{reid}. The main difference is that we here also take into consideration that the field $k$ is of positive characteristic. The proof uses formal derivatives, so for this purpose, we introduce the notation $f'_{x_j}$ for the formal derivative of a polynomial $f$, with respect to the variable $x_j$. 

\begin{lemma}\label{lemma:deriv}
 Suppose $f ,h \in k[x_1, \dots ,x_n]$ is such that $h^mf \in (x_1^{d_1}, \ldots, x_n^{d_n})$. Then 
 \[h^{m+1}f'_{x_j} \in (x_1^{\bar{d_1}}, \ldots , x_n^{\bar{d_n}}),\]
 where $\bar{d_j}=d_j-1$ and $\bar{d_i}=d_i$ for all $i \ne j$. 
\end{lemma}
\begin{proof}
 For simplicity we may assume that $x_j=x_1$, and denote the derivatives by just $f'$. We know that $h^mf=g_1x_1^{d_1} + \dots +g_nx_n^{d_n}$ for some $g_1, \ldots , g_n$. Now take the derivative of $h^mf$ with respect to $x_1$. We get
 \[
  mh^{m-1}h'f + h^mf' = g_1'x_1^{d_1}+d_1g_1x_1^{d_1-1} + g_2'x_2^{d_2} + \dots +g_n'x_n^{d_n}.
 \]
 After multiplication by $h$ we see that
 \[
 mh^{m}h'f + h^{m+1}f' \in (x_1^{d_1-1},x_2^{d_2} , \dots , x_n^{d_n}).
\]
Since $h^mf \in (x_1^{d_1}, \ldots, x_n^{d_n})$ we can conclude that $h^{m+1}f' \in (x_1^{d_1-1},x_2^{d_2} , \dots , x_n^{d_n}).$
\end{proof}

\begin{theorem}\label{thm:stanley}
 Let $A=k[x_1, \ldots, x_n]/(x_1^{d_1}, \ldots, x_n^{d_n})$, where $k$ is a field of characteristic $p >0$, let 
 $t = \sum_{i=1}^n (d_i - 1)$. Let $m$ be a positive integer such that $m+t <2p$. Then the map $A_i \to A_{i+m}$ given by
 \[
  f \mapsto f \cdot (x_1+ \dots +x_n)^m 
 \]
is injective for $i \le (t-m)/2$.
\end{theorem}
\begin{proof}
Let $s = x_1+\cdots + x_n$. Suppose that $f$ is a nonzero homogeneous element such that $f \cdot s^m=0$. We shall prove that $\deg f >(t-m)/2$ by induction over the degree of $f$. Assume first that $\deg f =0$. In this case $s^m=0$. We want to show that $(t-m)/2<0$, which is to say that $m>t$. Suppose, for a contradiction, that $m \le t$. Since $m+t<2p$ we get $2m< 2p$, i. e. $m<p$. Let us now look at the expansion of $s^m$. Since $m \le t$ we can find a monomial $x_1^{\alpha_1} \cdots x_n^{\alpha_n}$ in the expansion, with $\alpha_i< d_i$ for all $i$. The coefficient of this monomial is not divisible by $p$, because $m<p$. But this contradicts $s^m=0$. Hence $m>t$.
 
Now assume $\deg f >0$, and that the claim is true for all homogeneous polynomials of lower degree. If $\deg f \ge p$ we are already done, since $p > (t-m)/2$ by the assumption on $m$. Therefore we may assume that $\deg f <p$. Let $F$ be a homogeneous element in $k[x_1, \ldots, x_n]$, such that the image of $F$ in $A$ is $f$. Since $\deg F = \deg f$ we can find a variable $x_j$ such that the image of $F'_{x_j}$ in $A$ is nonzero. By Lemma \ref{lemma:deriv} the image of $s^{m+1}F'$ is zero in $k[x_1, \ldots, x_n]/(x_1^{\bar{d_1}}, \ldots , x_n^{\bar{d_n}})$, where $\bar{d_j}=d_j-1$ and $\bar{d_i}=d_i$ for all $i \ne j$. Note that we have increased $m$ by one, and decreased $t$ by one. The condition $(m+1)+(t-1)=m+t<2p$ is still true. By the inductive assumption      
\[ \deg f = \deg F = \deg F' +1 > \frac{(t-1)-(m+1)}{2}+1 = \frac{t-m}{2}. \]
\end{proof}

\begin{remark}
 The element $x_1+ \dots +x_n$ in Theorem \ref{thm:stanley} above can be replaced by any linear form $c_1x_1+ \dots +c_nx_n$, where all the $c_i$'s are nonzero. 
\end{remark}

\begin{remark}
 There is a small mistake in the proof of \cite[Theorem 5]{reid}. The derivative is taken w.r.t. $x_n$, where $d_n = \max(d_1, \ldots, d_n)$. But with this choice of variable, the derivative of $F$ might be 0, and the proof fails. For example $(x_1+x_2+x_3)^3(x_1-x_2)=0$ in $k[x_1,x_2,x_3]/(x_1^2,x_2^2,x_3^3)$, but the derivative of $x_1-x_2$ w. r. t. $x_3$ is $0$. The problem is solved by taking the derivative w.r.t. another variable.
\end{remark}

Note that for $m>t$ the theorem is trivial, because there is nothing to prove. In that case we don't need any condition on the characteristic of $k$. It is also possible to formulate a condition on the $d_i$'s, instead of $p$, as we will see in the following theorem.

\begin{theorem} \label{thm:stanleydi}
Let $A=k[x_1, \ldots, x_n]/(x_1^{d_1}, \ldots, x_n^{d_n})$, where $k$ is any field and $t = \sum_{i=1}^n (d_i - 1)$. If $\max(d_1, \ldots, d_n)>(t+m)/2$ for an integer $m$, then the map $A_i \to A_{i+m}$ given by 
\[f \mapsto f \cdot (x_1+ \dots + x_n)^m\] 
is injective for all $i \le (t-m)/2$. 
\end{theorem}
\begin{proof}
Let $f$ be a homogeneous polynomial such that $(x_1+ \dots + x_n)^mf=0$ in $A$. Then 
\[
 (x_1+ \dots + x_n)^mf=g_1x_1^{d_1} + \dots + g_nx_n^{d_n} \text{ in } k[x_1,\ldots,x_n]. 
\]
Without loss of generality, we can assume that $d_1 = \max(d_1, \ldots, d_n)$. We can see that $g_1 \ne 0$ in the following way. Let $cx_1^{\alpha_1} \cdots x_n^{\alpha_n}$ be the term in $f$ with highest $x_1$-degree, and all $\alpha_i < d_i$. When multiplying by $(x_1+ \dots + x_n)^m$ we get the term $cx_1^{\alpha_1+m} x_2^{\alpha_2} \cdots x_n^{\alpha_n}$, which can not be cancelled by any other term. Hence $\alpha_1+m \ge d_1$, and we can conclude that 
\[\deg f \ge d_1-m>\frac{t-m}{2}.\]
\end{proof}

We have now seen conditions for the map $A_i \to A_{i+m}$ given by $a \mapsto a \cdot s^m$ to be injective for $i \le (t-m)/2$. For completeness we shall now prove that it follows that the map is surjective for larger $i$. We remark that this result is not new, it can for instance be found in \cite{reid}. The proof is based on the fact that $A_i \times A_{t-i} \to A_t \cong k$ is a perfect pairing.

\begin{proposition}\label{prop:maxrank} 
Let $A=k[x_1, \ldots, x_n]/(x_1^{d_1}, \ldots, x_n^{d_n})$ and $t=\sum_{i=1}^n(d_i-1)$, and let $s\in A$ be a form of degree $d$. The map $A_i \to A_{i+d}$ given by $f \mapsto f \cdot s$ has maximal rank for each $i$ if and only if it is injective for all $i \le (t-d)/2$. 
\end{proposition}
\begin{proof}
Suppose that the above map has maximal rank for each $i$, and let $i \le (t-d)/2$. Recall that the Hilbert function is symmetric about $t/2$ and weakly increasing up to $t/2$. If $i+d \le t/2$, then $H(i) \le H(i+d)$, and the multiplication map should be injective. Suppose $i+d \ge t/2$. Since $i \le (t-d)/2$ we have that $t-i \ge i+d$. We use that the Hilbert function is weakly decreasing after $t/2$ and get $H(i)=H(t-i) \le H(i+d)$. Hence the multiplication map is injective.  
 
Let the multiplication maps be denoted by $\cdot s : A_{i} \to A_{i+d}$. Suppose that $\cdot s : A_{i} \to A_{i+d}$ is injective, for some fixed $i$. We shall prove that $\cdot s : A_{t-i-d} \to A_{t-i}$ is surjective, which then completes the proof. Let $\{ y_1, \ldots, y_N\}$ be the monomial basis for $A_{i}$, as a vector space over $k$. Since $\cdot s : A_{i} \to A_{i+d}$ is injective, $sy_1, \ldots, sy_N$ are linearly independent, and can be extended to a basis $\{ sy_1, \ldots, sy_N, \xi_{N+1} , \ldots, \xi_{M} \}$ of $A_{i+d}$. Also, if we let $m=x_1^{d_1-1} \!\cdots x_{n}^{d_n-1}$, then $\{m\}$ is a basis for $A_t$, and $\{\frac{m}{y_1}, \ldots, \frac{m}{y_N} \}$ a basis for $A_{t-i}$. Now, take some $f \in A_{t-i}$. Notice that, for every $j=1, \ldots, N$ we have $fy_j =cm$, where $c$ is the coefficient of $\frac{m}{y_j}$ in $f$. We want to prove that there is some $g \in A_{t-i-d}$ such that $f=sg$, and we will do that by proving $fy_j=sgy_j$ for every $j$. To find this $g$ we first need to define a linear map $\psi: A_{i+d} \to A_{t}$ by $\psi(sy_j)=fy_j$, and $\psi(\xi_\ell)=0$. Let $\{z_1, \ldots, z_M \}$ be the monomial basis for $A_{i+d}$, and suppose $\psi(z_j)=\alpha_jm$. Put 
$g=\alpha_1 \frac{m}{z_1} + \dots + \alpha_M\frac{m}{z_M} \in A_{t-i-d}. $
Then $gz_j=\alpha_jm$, and hence the map $\psi$ is given by multiplication by $g$. Since $\psi(sy_j)=fy_j$, we get that $fy_j=sgy_j$ for every $j$. This proves that the map $\cdot s : A_{t-i-d} \to A_{t-i}$ is surjective.
\end{proof}

Theorem \ref{thm:stanley}, Theorem \ref{thm:stanleydi}, and Proposition \ref{prop:maxrank} can now be combined into the following theorem, which we will use to derive results on the SLP.  

\begin{theorem}\label{thm:stanleymrp}
Let $A = k[x_1,\ldots,x_n]/(x_1^{d_1}, \ldots, x_n^{d_n})$, where $k$ is a field of characteristic $p>0$, and $t = \sum_{i = 1}^n (d_i-1).$ If $\max(p,d_1, \ldots, d_n)>(t+m)/2$ for an integer $m$, then the map $A_i \to A_{i+m}$ given by 
\[ f \mapsto f \cdot (x_1 + \dots + x_n)^m \]
has maximal rank for each $i$. 
\end{theorem}

\begin{remark}
We were noticed about the fact that we are not the first to use the result in \cite{reid} in the positive characteristic case --- 
Vraciu \cite{vraciu} has used similar techniques in order to obtain results on the minimal degree of a non-trivial zero-divisor in $k[x_1, \ldots, x_n]/(x_1^{d_1}, \ldots, x_{n-1}^{d_{n-1}}, (x_1+\cdots + x_{n-1})^{d_n})$ under certain assumptions on the $d_i$'s.
\end{remark}

\section{The strong Lefschetz property} \label{sec:slp}

\begin{definition} Let $A$ be a graded artinian algebra. We say that $A$ has the \emph{strong Lefschetz property} (SLP) if there is  a linear form $\ell$ in $A$ such that the map $A_i \to A_{i+m},$ given by $a \mapsto \ell^m \cdot a$, has maximal rank for all $i$ and all $m\geq 1$. In this case, $\ell$ is said to be a \emph{strong Lefschetz element}.
\end{definition}

Let $A=k[x_1, \ldots, x_n]/(x_1^{d_1}, \ldots , x_n^{d_n})$ with $k$ of characteristic $p>0$ and $t=\sum_{i=1}^n(d_i-1)$. It can easily be seen that when $n=1$, we have the SLP regardless of the characteristic. When $n=2$ the situation is more involved, as is indicated in Theorem 4.9 in \cite{cook}. We aim to give a complete characterization of when $A$ has the SLP, for $n \ge 3$, and leave the two variable case out of this paper.

When $p>t$, obviously $p>(t+m)/2$ for all $m \le t$. Then we can use Theorem \ref{thm:stanleymrp} to conclude that $x_1+\dots +x_n$ is a strong Lefschetz element in $A$. For $m>t$, any map $A_i \to A_{i+m}=0$ has maximal rank, and there is nothing to prove. Hence we have the SLP when $p>t$. This result was also proved by Cook II in \cite[Theorem 3.6]{cook}, and earlier for two variables by Lindsey in \cite[Lemma 5.2]{lind}. 

In \cite{cook} it is also given a complete characterization of when $A$ has the SLP, for $k$ infinite  of characteristic 2, and for $k$ infinite of positive characteristic and $d_1= \dots = d_n$. It is conjectured, \cite[Conjecture 7.6]{cook}, that $A$, where $\max(d_1, \ldots, d_n) \le (t+1)/2$, has the SLP if and only if the characteristic of $k$ is zero or greater than $t$. As we have seen, the case of characteristic zero was proved by Stanley in \cite{stanley}. Notice that the condition $\max(d_1, \ldots, d_n) \le (t+1)/2$ implies $n \ge 3$. The classification that we obtain in Theorem \ref{thm:slp} will settle the conjecture as a special case.

\subsection{Necessary conditions for the SLP}

The key for the necessary condition for $A$ having the SLP is the following two lemmas.

\begin{lemma} \label{lemma:largepower}
Let $A = k[x_1,\ldots,x_n]/(x_1^{d_1}, \ldots, x_n^{d_n})$, where $k$ is a field of characteristic $p>0$. 
Write $d_i = N_i p + r_i$, with $0 < r_i \le p$. Let $N = \sum_{i=1}^n N_i$.  

Let $0\le j \le n$ and $m=N-j+1$. Then
$$\ell^{mp} \cdot x_1^{r_1} x_2^{r_2} \cdots x_j^{r_j} = 0$$ 
in $A$, for any linear form $\ell$. The monomial $x_1^{r_1} \cdots x_j^{r_j}$ is interpreted as $1$ when $j=0$.
\end{lemma}

\begin{proof}
Let $\ell=c_1x_1+ \dots +c_nx_n$, with $c_i \in k$ for $i = 1, \ldots, n$. Since $k$ is a field of characteristic $p$, it holds that
\[
 (c_1x_1+ \dots +c_nx_n)^{mp}=(c_1^p x_1^p+ \dots +c_n^p x_n^p)^{m} = \sum_{\alpha} \hat{c}_{\alpha} x_1^{\alpha_1 p} \cdots x_n^{\alpha_n p},
\]
where the sum goes over all $n$-tuples $\alpha=(\alpha_1, \ldots, \alpha_n)$ such that $\sum_{i=1}^n\alpha_i=m$. If $m>N$, then at least one $\alpha_i > N_i$, in every $n$-tuple $\alpha$. Then $\alpha_ip \ge (N_i+1)p \ge d_i$, and hence $(c_1x_1+ \dots +c_nx_n)^{mp}=0$. 

Assume now that $m \le N$. We need only to include $n$-tuples $\alpha=(\alpha_1, \ldots, \alpha_n)$, where each $\alpha_i \le N_i$, because only these give non-zero terms in the sum. But then we can find at least one $\alpha_i=N_i$ among $\alpha_1, \ldots, \alpha_j$, in each $n$-tuple $\alpha$, because otherwise $\sum_{i=1}^n \alpha_i \le N-j=m-1$. Thus $x_i^{\alpha_ip}\cdot x_i^{r_i}=x_i^{d_i}=0$ and
\[
 (c_1 x_1+ \dots +c_n x_n)^{mp}=\left(\sum_{\alpha} \hat{c}_{\alpha} x_1^{\alpha_1 p} \cdots x_n^{\alpha_n p}\right)\cdot x_1^{r_1} \cdots x_j^{r_j}=0
\]
\end{proof}

\begin{example}
 Let $A=k[x_1,x_2,x_3]/(x_1^{12},x_2^9,x_3^3)$ with $k$ of characteristic $5$. Then $d_1=2\cdot 5 +2$, $d_2=5+4$,  $d_3=3$, and $N=3$. Let $j=1$ in Lemma \ref{lemma:largepower}. Then $m=3$, and $\ell^{3 \cdot 5}x_1^2=0$, for any linear form $\ell$. We have $t=11+8+2=21$, and $\deg(x_2^2) \le (21-3\cdot 5)/2$. It follows from Proposition \ref{prop:maxrank} that multiplication by $\ell^{15}$ can not have maximal rank, and $A$ does not have the SLP. 
 
 We could also choose $j=2$ in Lemma \ref{lemma:largepower}. Then we get that $\ell^{2 \cdot 5}x_1^2x_2^4=0$. This is not a proof that $A$ fails to have the SLP, because $\deg(x_1^2x_2^4) \not \le (21-2 \cdot 5)/2$. 
 
 If we choose $j=3$, Lemma \ref{lemma:largepower} says that $\ell^5 x_1^2x_2^4x_3^3=0$. But this is trivial, since $x_1^2x_2^4x_3^3=0$.
\end{example}

As we saw in the example above, Lemma \ref{lemma:largepower} can be used to prove that an algebra fails to have the SLP, but we need to make sure that the monomial $x_1^{r_1} \cdots x_j^{r_j}$ is nonzero, and of degree small enough. We collect the details of this in the next lemma. Notice also that the indices $1, 2 , \dots, j$ can be replaced by any suitable choice of $j$ indices.

\begin{lemma}\label{lemma:notslp}
Let $A = k[x_1,\ldots,x_n]/(x_1^{d_1}, \ldots, x_n^{d_n})$, where $k$ is a field of characteristic $p>0$. Write $d_i = N_i p + r_i$, with $0 < r_i \le p$. Let $\Lambda$ be an index set such that $d_i>r_i$ for all $i \in \Lambda$, and let $m=\sum_{i=1}^nN_i-|\Lambda|+1$. If 
\[\sum_{i \in \Lambda}r_i \le \frac{\sum_{i=1}^n(d_i-1)-mp}{2}, \]
or equivalently 
\[
 \sum_{i \in \Lambda}r_i \le \sum_{i \notin \Lambda} r_i -n+(|\Lambda|-1)p,
\]
then $A$ fails to have the SLP. 
\end{lemma}
\begin{proof}
Let $\Lambda$ be an index set as above. Since $r_i <d_i$ for all $i \in \Lambda$, and $f=\prod_{i \in \Lambda} x_i^{r_i} \neq 0$. By Lemma \ref{lemma:largepower} we have $\ell^{mp}f=0$ for any nonzero linear form $\ell$. Then the map $A_{d} \to A_{d+mp}$, given by multiplication by $\ell^{mp}$, is not injective for $d=\sum_{i \in \Lambda}r_i$. If
\[\sum_{i \in \Lambda}r_i \le \frac{\sum_{i=1}^n(d_i-1)-mp}{2}, \]
if follows from Proposition \ref{prop:maxrank} that we can not find a strong Lefschetz element.
\end{proof}

From now on, we will assume that $n \ge 3$ in $A = k[x_1,\ldots,x_n]/(x_1^{d_1}, \ldots, x_n^{d_n})$. Thus we should also assume $d_i \ge 2$ for all $i$. Otherwise $A$ might be isomorphic to a ring with only one or two variables.

\begin{proposition} \label{prop:slpnec}
Let $A = k[x_1,\ldots,x_n]/(x_1^{d_1}, \ldots, x_n^{d_n})$, where $n \geq 3$, $d_1 \ge d_2 \ge \dots \ge d_n \ge 2$, and where $k$ is a field of characteristic $p > 0$. Write $d_i = N_i p + r_i$. If one of the conditions below holds, then $A$ fails to have the SLP.

\begin{enumerate}
\item $p = 2$,
\item $p \geq 3$, $d_1 > p, d_2 \leq p$ and $\sum_{i=2}^n (d_i - 1) > r_1$,
\item $p \geq 3$, $d_1 > p, d_2 \leq p$ and $r_1 + \sum_{i=2}^n(d_i-1) > p$,
\item $p \geq 3$, $d_1 \leq p$ and $\sum_{i=1}^n(d_i-1) \geq p$, 
\item $p \geq 3$ and $d_2 > p$.
\end{enumerate}
\end{proposition}

\begin{proof}

\begin{enumerate}
\item If $k$ is an infinite field, then  $A$ fails to have the SLP \cite[Corollary 6.3]{cook}. Suppose that $k$ is finite. Let $k'$ be an infinite field such that $k \subset k'$. If $A$ has the SLP, so does $k'[x_1,\ldots,x_n]/(x_1^{d_1}, \ldots, x_n^{d_n})$. Thus $A$ fails to have the SLP. 

\item Notice that $d_i = r_i$ for $i = 2, \ldots, n$. Let $\Lambda = \{1\}$. Then 
\[\sum_{i \notin \Lambda} r_i -n+(|\Lambda|-1)p = \sum_{i=2}^n d_i-n= \sum_{i=2}^n (d_i - 1) - 1 \ge r_1,\] so by Lemma \ref{lemma:notslp}, $A$ fails to have the SLP. 

\item Let $\Lambda = \{ \}$. Then 
\[\sum_{i \notin \Lambda} r_i - n + (|\Lambda|-1)p = r_1 +\sum_{i=2}^n (d_i-1) -1 -p \geq p+1 - 1 - p = 0,\]
so by Lemma \ref{lemma:notslp}, $A$ fails to have the SLP. 
\item Let $\Lambda = \{ \}$. Then 
\[\sum_{i \notin \Lambda} r_i - n + (|\Lambda|-1)p = \sum_{i=1}^n (r_i-1) - p=\sum_{i=1}^n (d_i-1) - p \geq 0,\] so by Lemma \ref{lemma:notslp}, $A$ fails to have the SLP. 
\item 

Let us first assume that $r_i \ge 2$ for all $i$. We will not use the order between $d_1$ and $d_2$, so we may assume that $r_1 \le r_2$. By taking $\Lambda=\{1\}$ in Lemma \ref{lemma:notslp} we get that $A$ fails to have the SLP when the inequality
\begin{equation} \label{lambd1}
r_1 \le r_2+ \dots + r_n-n 
\end{equation} 
holds. The inequality holds in all cases, except when $n=3$, $r_1=r_2$, and $r_3=2$. When taking $\Lambda=\{1,2\}$ in Lemma \ref{lemma:notslp} we get the inequality
\begin{equation}\label{lambd12}
r_1+r_2 \le r_3+ \dots + r_n-n+p.
\end{equation}
In the case when $n=3$, $r_1=r_2$, and $r_3=2$, this becomes $r_1+r_2\le p-1$. Thus $A$ fails to have the SLP when $r_1+r_2\le p-1$. 

Let $\Lambda = \{ \}$. Then $\sum_{i \notin \Lambda} r_i - n + (|\Lambda|-1)p = \sum_{i=1}^nr_i - p$, so if $\sum_{i=1}^nr_i \geq p$, then $A$ fails to have the SLP by Lemma \ref{lemma:notslp}. In our case, this means that  $A$ fails to have the SLP if  $r_1 + r_2 + r_3 - 3 \geq p \Leftrightarrow r_1 + r_2 \geq p + 1$. The only case that is not covered here is when $r_1+r_2=p$. However, this can not happen when $r_1=r_2$ and $p$ is an odd prime. 

Next, assume that there is at least one $r_i=1$. Let $\Lambda=\{i ~| r_i=1\}$. Notice that $r_i=1$ is only allowed when $N_i>0$, otherwise we get $d_i=1$. Thus this is an appropriate choice of the set $\Lambda$, in Lemma \ref{lemma:notslp}. In this case $\sum_{i \in \Lambda}r_i=|\Lambda|$, and the inequality from Lemma \ref{lemma:notslp} becomes
\[|\Lambda| \le \sum_{i \notin \Lambda} r_i -n+(|\Lambda|-1)p.\]
Since $\sum_{i \notin \Lambda} r_i \ge 2(n-|\Lambda|)$ the above inequality is true if
\[
 |\Lambda| \le 2(n-|\Lambda|) -n+(|\Lambda|-1)p.
\]
This can be rewritten as $|\Lambda|(p-3)+n-p\ge0$. We know that $|\Lambda| \ge 1$ and $p-3 \ge 0$, so $|\Lambda|(p-3)+n-p\ge p-3+n-p=n-3 \ge 0$ is true. This shows that we do not have the SLP.

\end{enumerate}
\end{proof}

\subsection{Sufficient conditions for the SLP}
We now turn into the sufficient conditions for the SLP.

\begin{proposition} \label{prop:slpcase1}
Let $A = k[x_1,\ldots,x_n]/(x_1^{d_1}, \ldots, x_n^{d_n})$, where $k$ is a field of characteristic $p>0$. Suppose $d_1,\ldots,d_n \geq 2$. Let $t=\sum_{i=1}^n(d_i-1)$. Then $A$ has the SLP if $t < p$.

\end{proposition}

\begin{proof}
We want to show that there is a linear form $s$ such that multiplication by $s^m$ has maximal rank in every degree. If $m>t$ there is nothing to prove, so assume $m \le t$. Then $t+m \le 2t <2p$, and by Theorem \ref{thm:stanleymrp}, the linear form $x_1+ \dots +x_n$ is a strong Lefschetz element.  
\end{proof}

There is one more case when $A$ has the SLP, as we will see in the next proposition.  

\begin{proposition}\label{prop:slpcase2}
Let $A = k[x_1,\ldots,x_n]/(x_1^{d_1}, \ldots, x_n^{d_n})$, where $k$ is a field of characteristic $p>0$. Let $d_1=\max(d_1, \dots, d_n)$, and write $d_1=N_1p+r_1$ where $0<r_1 \le p$. Then $A$ has the SLP if all the following conditions are satisfied
\begin{enumerate}
\item $d_1 > p$ and $d_i \le p$ for $i=2, \dots , n$,
\item $\displaystyle r_1+ \sum_{i=2}^n (d_i-1)\le p$,
\item $\displaystyle \sum_{i=2}^n(d_i-1) \le r_1$.
\end{enumerate}
\end{proposition}
The proof of this proposition uses the same technique as the proof of Theorem \ref{thm:stanley}. Notice that, with the notation from Lemma \ref{lemma:notslp}, $d_i=r_i$ here, for $i=2, \ldots, n$. 
\begin{proof}
Let $t=\sum_{i=1}^n(d_i-1)$, and $s=x_1 + \dots + x_n$. We want to prove that $s$ is a strong Lefschetz element, that is, for any positive integer $m$ the map $A_i \to A_{i+m}$, given by multiplication by $s^m$, has maximal rank. By Proposition \ref{prop:maxrank} it is enough to show that it is injective for $i \le (t-m)/2$. 
Suppose that there is a nonzero homogeneous polynomial $f$ such that $s^mf=0$. We want to show that $\deg f >(t-m)/2$. The proof is by induction on the degree of $f(1,x_2, \ldots, x_n)$. 

Suppose first that $\deg(f(1,x_2, \ldots, x_n))=0$. Then $f$ is a monomial in the variable $x_1$ only, so $f= c x_1^a$, and without loss of generality we may assume that $c=1$. We have $s^mx_1^a=0$. Let $m=Np+r$, where $0 \le r <p$. Then 
\[
 s^mx_1^a=s^{Np}s^rx_1^a=(x_1^p+ \dots + x_n^p)^Ns^rx_1^a = x_1^{Np}s^rx_1^a
\]
by the first condition. Hence $x_1^{Np}s^rx_1^a=0$. Note that, when expanding $s^r$, we get nonzero coefficients because $r<p$. 

In the case when $r \le \sum_{i=2}^n(r_i-1)$, we can find a nonzero term in the expansion of $s^r$ which does not contain the variable $x_1$. Then, for $s^mx_1^a$ to be zero, we must have $Np+a \ge d_1=N_1p+r_1$. We rewrite this as $a \ge (N_1-N)p+r_1$. We want to show that $a >(t-m)/2$, which follows if $(N_1-N)p+r_1>(t-m)/2$. Since
\begin{align*}
 t-m=&\sum_{i=1}^n(d_i-1)-m \\ 
 =& N_1p+r_1 -1 + \sum_{i=2}^n(r_i-1)-(Np+r)=(N_1-N)p+ \sum_{i=1}^n(r_i-1)-r
\end{align*}
the inequality that we want to prove becomes 
\[
 2((N_1-N)p+r_1) > (N_1-N)p+ \sum_{i=1}^n(r_i-1)-r 
 \]
 or equivalently
 \[
 (N_1-N)p+r_1 > \sum_{i=2}^n r_i -n -r.
\]
This is true when $N_1 \ge N$, by our third condition. If $N_1<N$ then 
\[
 m \ge Np \ge N_1p+p>N_1p+ \sum_{i=1}^n(r_i-1)=t
\]
by our second condition. Then $A_{i+m}=0$ for any $i$, and there is nothing to prove. 

In the other case, when $r > \sum_{i=2}^n(r_i-1)$, the term of lowest $x_1$-degree in $x_1^{Np}s^rx_1^a$ has $x_1$-degree 
\[
 r-\sum_{i=2}^n(r_i-1)+Np+a \ge d_1.
\]
Then
\[
 a \ge d_1+\sum_{i=2}^n(r_i-1)-Np-r=t+1-m >\frac{t-m}{2}.
\]
This finishes the induction base. 

For the induction step, assume that $\deg (f(1, x_2, \ldots, x_n))>0$, and that the statement holds for homogeneous polynomials of lower degree. Then $f$ contains at least one variable $x_j$, where $j\neq 1$ and $d_j>1$. We know that the $x_j$-degree of $f$ is less than $d_j \le p$. Let $F$ be a homogeneous polynomial in $k[x_1, \ldots , x_n]$ such that the image of $F$ in $A$ is $f$. Then the image of $F'_{x_j}$ is nonzero. By Lemma \ref{lemma:deriv} the image of $s^{m+1}F'_{x_j}$ is 0 in $k[x_1, \ldots, x_n]/(x_1^{\bar{d_1}}, \ldots , x_n^{\bar{d_n}})$, where $\bar{d_j}=d_j-1$ and $\bar{d_i}=d_i$ for all $i \ne j$. The three conditions of this proposition hold also for this ring. By the inductive assumption      
\[ \deg f = \deg F = \deg F'_{x_j} +1 > \frac{(t-1)-(m+1)}{2}+1 = \frac{t-m}{2}. \]
We have now proved that $s$ is a strong Lefschetz element. 
\end{proof}

\subsection{The classification of the SLP when $ n\geq 3$}
We now combine the results in the two previous sections to obtain a classification of the SLP when 
$ n\geq 3$. The conjecture by Cook II on the SLP \cite[Conjecture 7.6]{cook} follows as a special case of Theorem \ref{thm:slp}.

\begin{theorem}\label{thm:slp}
Let $A=k[x_1, \ldots, x_n]/(x_1^{d_1}, \ldots, x_n^{d_n})$ where $n \ge 3$, $d_i \ge 2$ for all $i$, and $k$ is a field of characteristic $p>0$. Let $t = \sum_{i = 1}^n (d_i-1)$ and
let $d_1=\max(d_1, \dots, d_n)$. Write $d_1=N_1p+r_1$ with $0<r_1\le p$.
Then $A$ has the SLP if and only if one of the following two conditions hold
\begin{enumerate}
\item $t <p$,
\item $d_1 > p$, $d_i \leq p$ for $i =2, \ldots, n$ and $\sum_{i=2}^n(d_i-1) \le \min(r_1,p-r_1)$.
\end{enumerate} 
\end{theorem}

\begin{proof}
The above conditions is exactly the conditions from Proposition \ref{prop:slpcase1} and Proposition \ref{prop:slpcase2}. We want to prove that the SLP fails in all other cases. Notice first that when $p=2$, condition 2 in Proposition \ref{prop:slpcase2} fails. So does the condition in Proposition \ref{prop:slpcase1}, since $t \ge 3$ when $n \ge 3$ and all $d_i \ge 2$. This agrees with Proposition \ref{prop:slpnec}, which states that we can never have the SLP when $p=2$. Let us assume $p \ge 3$ for the remainder of this proof.

Let $d_1 \geq \cdots \geq d_n$ be such that $A$ does not fulfill the condition of Proposition \ref{prop:slpcase1}, that is $t \ge p$. If condition 1 of Proposition \ref{prop:slpcase2} is not satisfied, we have one of the following two cases.
\begin{itemize}
 \item $t \ge p$ and $d_1 \le p,$
 \item $t \ge p$ and $d_2 >p$.
\end{itemize}
In the latter, the inequality $t \geq p$ is superfluous. By Proposition \ref{prop:slpnec} $A$ fails to have the SLP in these two cases. Suppose instead that the condition 1 of Proposition \ref{prop:slpcase2} is satisfied, but that condition 2 or 3 fails. This gives the two cases
\begin{itemize}
\item $t \geq p$, $d_1 > p$, $d_2 \leq p$ and $\sum_{i=2}^n(d_i-1) >r_1,$
\item $t \geq p$, $d_1 > p$, $d_2 \leq p$ and $\sum_{i=2}^n(d_i-1) >p-r_1.$ 
\end{itemize}
Also here, the inequality $t \geq p$ is superfluous and $A$ fails to have the SLP by Proposition \ref{prop:slpnec}.
\end{proof}

 Let us now turn to \cite[Conjecture 7.6]{cook}, namely that if 
 $\max(d_1, \ldots, d_n) \le (t+1)/2$, then $A$ has the SLP if and only if the characteristic of $k$ is zero or greater than $t$. 
 
 We already know that $A$ has the SLP in characteristic zero and by Theorem \ref{thm:slp}, $A$ has the SLP when $p>t$. Suppose that the characteristic is $p \leq t$. We can assume that $d_1 = \max(d_1,\ldots,d_n)$. Then $d_1 \leq (t+1)/2 \Leftrightarrow 2 d_1 \leq t + 1 \Leftrightarrow 2d_1 \leq \sum_{i=1}^n (d_i -1) + 1\Leftrightarrow d_1 \leq \sum_{i=2}^n (d_i-1)$. By Theorem \ref{thm:slp}, in order for $A$ to have the SLP in case $p \leq t$, it must hold that $d_1 > p$. But then $p < \sum_{i=2}^n (d_i-1)$, so 
 $\sum_{i=2}^n (d_i-1)> \min(r_1,p-r_1)$. Thus $A$ fails to have the SLP and we have settled the conjecture.

\section{The weak Lefschetz property}

\begin{definition} Let $A$ be a graded artinian algebra. We say that $A$ has the \emph{weak Lefschetz property} (WLP) if there is a linear form $\ell$ such that the map $A_i \to A_{i+1}$, given by $a \mapsto \ell \cdot a$, has maximal rank for all $i$. In this case, $\ell$ is said to be a \emph{weak Lefschetz element}.
\end{definition}

In the next section we will generalize earlier results on the WLP for algebras of the form $k[x_1, \ldots, x_n]/(x_1^{d_1}, \ldots, x_n^{d_n})$.
Due to this generalization, the WLP is now completely classified in the case of uniform $d_i$'s, that is, in the case when $d_1= \dots =d_n$, see Table \ref{table:wlp}.

In Section \ref{sec:sufWLP}, we will turn to sufficient conditions on the WLP for mixed degrees. 

\subsection{Generalizations of earlier results}

The previous results in the literature on the WLP in positive characteristic are mainly under the assumption that the residue field is infinite. When working in positive characteristic it is natural to consider finite fields. Our first result is Proposition \ref{prop:wlpfinite}, where we show that the WLP of $A = k[x_1,\ldots,x_n]/I$ is independent of the cardinality of the field in the case when $I$ is monomial.

%
%
%

\begin{lemma} \label{lemma:wlpfinitenew}
Let $k$ be a field and let $k'$ be an extension field of $k$. Let $I \subset k[x_1,\ldots,x_n]$ be an ideal, let $A = k[x_1,\ldots,x_n]/I$ and let $f$ be a form in 
$A$. Then multiplication by $f$ on $A$ has maximal rank in every degree if and only if multiplication by $f$ on $A' =A \otimes_k k'$ has maximal rank in every degree.
\end{lemma}

\begin{proof}
It is clear that if multiplication by $f$ on $A$ has maximal rank in every degree, so does multiplication by $f$ on $A'$. 

Suppose that multiplication by $f$ on $A'$ has maximal rank. Since $I \subset k[x_1,\ldots,x_n]$, this means that we have a set of of matrices with coefficients in $k$, all of which has full rank over $k'$. But then they also have full rank over $k$, that is, multiplication by $f$ on $A$ has maximal rank in every degree. 
\end{proof}

We can now generalize \cite[Proposition 2.2]{migliore} by dropping the assumption on $k$ being infinite.

\begin{proposition}\label{prop:wlpfinite}
Let $k$ be a field and let $k'$ be an extension field of $k$. Let $I \subset k[x_1,\ldots,x_n]$ be a monomial ideal. Then the following are equivalent.

\begin{enumerate}
\item $A:=k[x_1,\ldots,x_n]/I$ has the WLP.
\item $A' := A \otimes_k k'$ has the WLP.
\item $x_1+ \cdots + x_n$ is a weak Lefschetz element of $A$.
\item $x_1+ \cdots + x_n$ is a weak Lefschetz element of $A'$.
\end{enumerate}
\end{proposition}

\begin{proof}
Let $k''$ be an infinite field such that $k \subset k' \subseteq k''$.  By \cite[Proposition 2.2]{migliore}, $x_1+ \cdots + x_n$ is a weak Lefschetz element of $A''$ if and only if $A''$ has the WLP. By Lemma \ref{lemma:wlpfinitenew}, it follows that $x_1+ \cdots + x_n$ is a weak Lefschetz element of $A$ if and only if it is a weak Lefschetz element of $A''$. Hence $A$ has the WLP if and only if $A''$ has the WLP. Now repeat the same argument with $A$ replaced by $A'$.
\end{proof}


The WLP in three variables for uniform degrees $d_1 = d_2 = d_3=d$ and $k$ algebraically closed was classified by Brenner and Kaid \cite{brenner}. When $d_1 = \cdots = d_n=d$ with $n \geq 4$ and $k$ an infinite field, the WLP was classified by Kustin and Vraciu \cite{kustin}. Using Proposition \ref{prop:wlpfinite}, we can extend these results. Thus the WLP in positive characteristic is now classified for uniform degrees, see Table \ref{table:wlp} below.

\begin{table}[h!]
$$
\begin{array}{l|ll}
n & \text{Condition on the characteristic} \\
\hline
    2 & \text{Has the WLP independent of the characteristic \cite{harima}, see also} \\
    & \text{Remark \ref{rmk:n2}.}
    \\
    3,d \text{ even} & \text{fails to have the WLP if and only there exists a $k \in \mathbb{N}$ and an}\\
    & n \in \mathbb{N}_{+}
    \text{such that } \frac{3d}{6k+2} > p^n > \frac{3d}{6k+4},
      \\
      3,d \text{ odd} & \text{fails to have the WLP if and only there exists a $k \in \mathbb{N}$ and an}\\          
      & n \in \mathbb{N}_{+}
      \text{such that} \frac{3d-1}{6k+2} > p^n > \frac{3d+1}{6k+4}, \text{ cf. \cite{brenner} and Proposition \ref{prop:wlpfinite}}.
        \\
    4& \text{Has the WLP if and only if }  d = kq+r \text{ for integers } k,q,d \text{ with }\\
    & 1 \leq k \leq \frac{p-1}{2}, r \in \{\frac{q-1}{2}, \frac{q+1}{2}\} \text{ and } q = p^e \text{ for some non-negative} \\
    & \text{integer } t, \text{ cf. \cite{kustin} and Proposition \ref{prop:wlpfinite}.}
    \\
    \geq 5 & \text{Has the WLP if and only if } p>(n(d-1)+1)/2, \text{ cf. \cite{kustin} and} \\ &\text{Proposition \ref{prop:wlpfinite}.}
  \end{array}
   $$
\caption{The classification of the WLP in positive characteristic for uniform degrees.}
  \label{table:wlp}
  \end{table}

\subsection{Sufficient conditions for the presence of the WLP for mixed degrees} \label{sec:sufWLP}

For mixed degrees, the situation is far from being understood. A combinatorial characterization of the WLP for $n=3$ and $k$ infinite was given in \cite{li}. Again, by Proposition \ref{prop:wlpfinite}, we can extend this result to finite fields. In \cite{cook}, Cook II giva a sufficient condition for the presence of the WLP, under the assumptions $n \ge 3$, and $k$ an infinite field. He conjectured, \cite[Conjecture 7.4]{cook}, that this result can be slightly improved. This conjecture is true, and the field $k$ does not need to be infinite, as we will see in the theorem below.  

\begin{theorem} \label{thm:wlplimit}
Let $A = k[x_1,\ldots,x_n]/(x_1^{d_1}, \ldots, x_n^{d_n})$, where $k$ is a field of characteristic $p>0$ and $t = \sum_{i = 1}^n (d_i-1).$ If $\max(p,d_1, \ldots, d_n)>(t+1)/2$, then $A$ has the WLP.
\end{theorem}
\begin{proof}[First proof]
Combine \cite[Theorem 1.6 (II)]{vraciu} and \cite[Prop. 5.2]{mig_miro-roig}, and Proposition \ref{prop:wlpfinite}.
\end{proof}
\begin{proof}[Second proof]
Put $m=1$ in Theorem \ref{thm:stanleymrp}, and get that $x_1+ \dots + x_n$ is a weak Lefschetz element. 
\end{proof}

\begin{remark}
As the SLP implies the WLP, we have the WLP when the conditions in Theorem \ref{thm:slp} are satisfied. But when condition 1 holds, obviously $p>(t+1)/2$. When condition 2 holds, $d_1>\sum_{i=2}^n(d_i-1)$.  Adding $d_1$ to both sides, we see that this is equivalent to $d_1>(t+1)/2$.
\end{remark}

\begin{remark} \label{rmk:n2}
For $A=k[x_1,x_2]/(x_1^{d_1},x_2^{d_2})$, the condition $\max(d_1,d_2)>(t+1)/2$ is satisfied for any choice of $d_1$ and $d_2$. By Theorem \ref{thm:wlplimit}, it follows that this $A$ has the WLP, independent of the characteristic of $k$. This is a folklore result, but is mentioned since all the references to this result point to the paper \cite{harima}, which only treats characteristic zero.
\end{remark}

We will now go deeper into the mixed degree case. While we only focus on sufficient conditions for the presence of the WLP, we believe that the results, and especially Theorem \ref{thm:wlpfamily} below, is an important step towards understanding the WLP for mixed degrees. 

First a word on the notation. Let $A$ be a graded algebra over a field $k$ and let $f$ be a form of degree $d$. By definition, multiplication by $f$ on $A$ has maximal rank in every degree if and only if $\dim_{k} (A/(f))_i  = \max(\dim_k A_i-\dim_k A_{i-d},0)$ for $i \geq d$. This condition can be expressed in terms of truncated Hilbert series \cite{froberg} --- multiplication by $f$ on $A$ has maximal rank in every degree if and only if $(A/f)(t) = [(1-t^d)A(t)]$, where $A(t)$ denotes the Hilbert series of $A$ where $[\sum_{i\geq 0} a_i t^i]$ means truncate at the first negative coefficient.

We are now ready to state the following useful and probably well known lemma.
\begin{lemma}\label{lemma:wlp-maxrank}
 The algebra $A=k[x_1, \ldots, x_n]/(x_1^{d_1}, \ldots, x_n^{d_n})$ has the WLP if and only if multiplication by $(x_1+ \dots + x_{n-1})^{d_n}$ has maximal rank in every degree in $B=k[x_1, \ldots, x_{n-1}]/(x_1^{d_1}, \ldots, x_{n-1}^{d_{n-1}})$. 
\end{lemma}
\begin{proof}
 The algebra $A$ has the WLP if and only if $(A/(x_1+\cdots +x_n))(t) = [(1-t)A(t)]$, while 
multiplication by  $(x_1+ \dots + x_{n-1})^{d_n}$ on $B$ has maximal rank in every degree if and only if $(B/(x_1+\cdots +x_{n-1})^{d_n})(t) = [(1-t^{d_{n}})B(t)]$. 

Now
$A(t) = (1-t^{d_n})B[x_n](t)=(1-t^{d_n})B(t)/(1-t)$ and $A/(x_1+\cdots +x_n) \cong B/\left((x_1+\cdots+x_{n-1})^{d_{n}}\right)$, so 
$A$ has the WLP if and only if  $(A/(x_1+\cdots +x_n))(t) = [(1-t)A(t)] = [(1-t^{d_n})B(t)]$, that is, if and only if multiplication by $(x_1+ \dots + x_{n-1})^{d_n}$ on $B$ has maximal rank in every degree.

\end{proof}
\begin{theorem}\label{thm:wlpfamily}
 Let $k$ be a field of characteristic $p> 0$, and let $d_1, \ldots, d_n, a$ be positive integers such that $d_i \le p^a$ for $i=1, \ldots, n-2$. The algebra $A=k[x_1, \ldots, x_n]/(x_1^{d_1}, \ldots, x_n^{d_n})$ has the WLP if
 \[ B=k[x_1, \ldots, x_n]/(x_1^{d_1}, \ldots, x_{n-2}^{d_{n-2}}, x_{n-1}^{d_{n-1}+bp^a}, x_n^{d_n+bp^a}) \]
 has the WLP, for any (and hence all) positive integers $b$.  If, in addition, $d_{n-1}+d_n\ge\sum_{i=1}^{n-2}(d_i-1)$, $A$ has the WLP if and only if $B$ does. 
\end{theorem}
\begin{proof}
We will prove the theorem for $b=1$. When this is done we can replace $d_1, \ldots, d_n$ by $d_1, \ldots, d_{n-2}, d_{n-1}+p^a, d_n+p^a$, and apply the theorem once more. This shows that the theorem also is true for $b=2$, and in the same way for any positive integer $b$. 

Let 
 \[
  A'=\frac{k[x_1, \ldots, x_{n-1}]}{(x_1^{d_1}, \ldots, x_{n-1}^{d_{n-1}})} ~~\textrm{and} ~~ B'=\frac{k[x_1, \ldots, x_{n-1}]}{(x_1^{d_1}, \ldots, x_{n-2}^{d_{n-2}}, x_{n-1}^{d_{n-1}+p^a})},
 \]
 and put $t_{A'}=\sum_{i=1}^{n-1}(d_i-1)$ and $t_{B'}=\sum_{i=1}^{n-1}(d_i-1)+p^a$. Let $s=x_1+ \dots + x_{n-1}$. By Lemma \ref{lemma:wlp-maxrank} and Proposition \ref{prop:maxrank}, $A$ has the WLP if and only if multiplication by $s^{d_n}$ is injective for $i \le (t_{A'}-d_n)/2$, in $A'$. In the same way $B$ has the WLP if and only if multiplication by $s^{d_n+p^a}$ is injective for $i \le (t_{B'}-(d_n+p^a))/2$, in $B'$.
 
 Suppose first that $B$ has the WLP, and that there is a homogeneous nonzero $f\in A'$, such that $s^{d_n}f=0$ in $A'$. If we lift $f$ to the polynomial ring, this can be realized as 
 \[
  f \notin (x_1^{d_1}, \ldots, x_{n-1}^{d_{n-1}}), ~~ \textrm{and} ~~ s^{d_n}f \in (x_1^{d_1}, \ldots, x_{n-1}^{d_{n-1}}).
 \]
It follows that
\[
  x_{n-1}^{p^a}f \notin (x_1^{d_1}, ..., x_{n-2}^{d_{n-2}},x_{n-1}^{d_{n-1}+p^a}), ~~ \textrm{and} ~~ s^{d_n}x_{n-1}^{p^a}f \in (x_1^{d_1}, ..., x_{n-2}^{d_{n-2}}, x_{n-1}^{d_{n-1}+p^a}).
\]
Since $d_i \le p^a$ for $i=1, \ldots, n-2$ we have that $s^{p^a}=x_{n-1}^{p^a}$. The two statements above then says that $f \ne 0$, but $s^{d_n+p^a}f=0$, in $B'$. Since multiplication by $s^{d_n}$ is injective up to degree $[(t_{B'}-(d_n-p^{a}))/2]$ 
\[
 \deg f > \frac{t_{B'}-(d_n+p^a)}{2}=\frac{t_{A'}-d_n}{2}.
\]
It follows that $A$ has the WLP. 

Suppose now that $A$ has the WLP, and that there is a homogeneous nonzero $f\in B'$, such that $s^{d_n+p^a}f=0$ in $B'$. Again, this can be realized as
\[
  f \notin (x_1^{d_1}, ..., x_{n-2}^{d_{n-2}},x_{n-1}^{d_{n-1}+p^a}), ~~ \textrm{and} ~~ s^{d_n}x_{n-1}^{p^a}f \in (x_1^{d_1}, ..., x_{n-2}^{d_{n-2}}, x_{n-1}^{d_{n-1}+p^a}).
\]
It follows that $s^{d_n}f \in (x_1^{d_1}, \ldots, x_{n-1}^{d_{n-1}}).$ If $f \ne 0$ in $A'$, it follows from the fact that $A$ has the WLP that 
\[
 \deg f > \frac{t_{A'}-d_n}{2} =\frac{t_{B'}-(d_n+p^a)}{2}.
\]
This shows that $B$ has the WLP. Now we must also consider the case when $f=0$ in $A'$. This means that $f=gx_{n-1}^{d_{n-1}}$ in $B'$, where $g$ is not divisible by $x_{n-1}^{p^a}$. Especially, $g$ is not zero. Then $\deg f \ge d_{n-1}$. It follows that $B$ has the WLP, if we can prove that $d_{n-1}>(t_{B'}-(d_n+p^a))/2$. The inequality can be rewritten as $d_{n-1}+d_n\ge\sum_{i=1}^{n-2}(d_i-1)$, which is true by assumption. 
\end{proof}

When $\sum_{i=1}^{n-2}(d_i-1) \le d_{n-1}+d_n <p$, the algebra $k[x_1, \ldots, x_n]/(x_1^{d_1}, \ldots, x_n^{d_n})$ has the WLP, by Theorem \ref{thm:wlplimit}. From Theorem \ref{thm:wlpfamily} we now get a family of algebras $k[x_1, \ldots, x_n]/(x_1^{d_1}, \ldots, x_{n-2}^{d_{n-2}}, x_{n-1}^{d_{n-1}+bp^a}, x_n^{d_n+bp^a})$ with the WLP, which are not all covered by Theorem \ref{thm:wlplimit}.
\begin{example}
Let $A=k[x_1, \ldots, x_5]/(x_1^3, x_2^3, x_3^4, x_4^5, x_5^5)$, with $k$ of characteristic 11. Then $t=2+2+3+4+4=15$, and $p>(t+1)/2$, so $A$ has the WLP by Theorem \ref{thm:wlplimit}. Since $5+5>2+2+3$ we can apply Theorem \ref{thm:wlpfamily}. Thus $k[x_1, \ldots, x_5]/(x_1^3, x_2^3, x_3^4, x_4^{5+11b}, x_5^{5+11b})$ has the WLP, for any nonnegative integer $b$.   
\end{example}

In \cite[Lemma 3.3]{cook}, it is shown that if $k$ is an infinite field of characteristic $p>0$ and $d_n=\sum_{i=1}^{n-1}(d_i-1)$, then $k[x_1,\ldots,x_n]/(x_1^{d_1}, \ldots, x_n^{d_n})$ has the WLP if and only if $\binom{d_n}{d_1-1, \ldots, d_{n-1}-1}$ is not divisible by $p$. We will now extend this result. 

\begin{proposition}\label{prop:wlpmultinom}
Let $k$ be a field of characteristic $p>0$, and suppose $d_n=\sum_{i=1}^{n-1}(d_i-1)$. If $\binom{d_n}{d_1-1, \ldots, d_{n-1}-1}$ is not divisible by $p$, then the algebras
\[ A = k[x_1,\ldots,x_n]/(x_1^{d_1}, \ldots, x_n^{d_n}) ~ \text{and }  ~ B = k[x_1,\ldots,x_n]/(x_1^{d_1}, \ldots, x_{n-1}^{d_{n-1}}, x_n^{d_n-1}) \] 
both have the WLP. If $\binom{d_n}{d_1-1, \ldots, d_{n-1}-1}$ is divisible by $p$, the algebra $A$ fails to have the WLP.
\end{proposition}

A proof of the fact that $A$ has the WLP if and only if $\binom{d_n}{d_1-1, \ldots, d_{n-1}-1}$ is not divisible by $p$ obviously follows from combining the proof of Lemma 3.3 in \cite{cook} and Proposition \ref{prop:wlpfinite}. However, instead of only proving that $B$ has the WLP, we found it natural to give a full proof, with the use of Proposition \ref{prop:maxrank}.

Notice that $(t_{A}+1)/2 = 2d_n/2 = d_n$, so when  $p \leq d_n$, we cannot use Theorem   \ref{thm:wlplimit} to draw any conclusion on $A$ when it comes to the WLP. Similarly, we have 
$(t_B + 1)/2 = (2d_n-1)/2$, so when $p \leq d_n-1$, we are outside the sufficient conditions in Theorem   \ref{thm:wlplimit}.

\begin{proof}
Let $R=k[x_1,\ldots,x_n]/(x_1^{d_1}, \ldots, x_{n-1}^{d_n-1}, x_n^{d_n+1}).$
This algebra has the WLP, by Theorem \ref{thm:wlplimit}. We next consider the algebra $A$. The Hilbert function of $A$ is strictly increasing up to $d_n-1$, then $H_A(d_n-1)=H_A(d_n)$, and after $d_n$ it is strictly decreasing. Let $s=x_1+ \dots + x_n$. By Proposition \ref{prop:maxrank} we only need to consider the maps $\cdot s: A_{i} \to A_{i+1}$ where $i<d_n$, which should be injective in the case of WLP. It follows directly from the fact that $R$ has the WLP, that $\cdot s: A_{i} \to A_{i+1}$ is injective for $i<d_n-1$. Hence $A$ has the WLP if and only if  $\cdot s:A_{d_n-1} \to A_{d_n}$ is injective. Notice that $A_{d_n-1}=R_{d_n-1}$ and $A_{d_n}=[R/(x_n^{d_n})]_{d_n}$. The map $s:R_{d_n-1} \to R_{d_n}$ is injective, so it follows that $\cdot s:A_{d_n-1} \to A_{d_n}$ also is injective, if and only if $x_n^{d_n} \notin sR_{d_n-1}$. It is equivalent to prove that $x_n^{d_n} \neq 0$ in $R/(s)$, which  holds if and only if $(x_1+ \ldots +x_{n-1})^{d_n} \ne 0$ in $k[x_1, \ldots, x_{n-1}]/(x_1^{d_1}, \ldots, x_{n-1}^{d_{n-1}}).$ Since 
\[(x_1+ \ldots +x_{n-1})^{d_n} = \binom{d_n}{d_1-1, \ldots,d_{n-1}-1}x_1^{d_1-1} \cdots x_{n-1}^{d_{n-1}-1}~~ \textrm{in} ~ A, \]
it follows that $A$ has the WLP if and only if the above multinomial coefficient is nonzero in $k$. 

Next we assume that $\binom{d_n}{d_1-1, \ldots, d_{n-1}-1}$ is not divisible by $p$, and consider $B$. The Hilbert function of $B$ is strictly increasing up to $d_n-1$, and then strictly decreasing. The multiplication maps $\cdot s : B_{i} \to B_{i+1}$ inherits the injectivity from $A$, when $i<d_n-2$. To prove that $B$ has the WLP we must prove that the map $\cdot s : B_{d_n-2} \to B_{d_n-1}$ also is injective. We know that $\cdot s : A_{d_n-2} \to A_{d_n-1}$ is injective, and $B_{d_n-2} = A_{d_n-2}$, and $B_{d_n-1} = [A/(x_n^{d_n-1})]_{d_n-1}$. The injectiveness follows in the same way as above, if we can prove that $(x_1+ \ldots +x_{n-1})^{d_n-1} \ne 0$ in $k[x_1, \ldots, x_{n-1}]/(x_1^{d_1}, \ldots, x_{n-1}^{d_{n-1}}). $ Now,
\begin{align*}
(x_1+ &\ldots +x_{n-1})^{d_n-1} =\\
\sum_{i=1}^{n-1} &\binom{d_n-1}{d_1-1, ..., d_{i-1}-1, d_i-2, d_{i+1}-1, ..., d_{n-1}-1} \frac{x_1^{d_1-1} \cdots x_{n-1}^{d_{n-1}-1}}{x_i}, 
\end{align*}
is nonzero if one of the coefficients is nonzero. The sum of these multinomial coefficients over $\mathbb{Z}$ is
\begin{align*}
\sum_{i=1}^{n-1}  &   \binom{d_n-1}{d_1-1, ..., d_{i-1}-1, d_i-2, d_{i+1}-1, ..., d_{n-1}-1}   =  \\  &=   \sum_{i=1}^{n-1} \binom{d_n}{d_1-1, \ldots, d_{n-1}-1}\frac{d_i-1}{d_n}=\binom{d_n}{d_1-1, \ldots, d_{n-1}-1}, 
\end{align*}
which is not divisible by $p$. Then the terms in the sum can not all be divisible by $p$, and hence one of the terms in the expansion of $(x_1+ \ldots +x_{n-1})^{d_n-1}$ is nonzero. This shows that $\cdot s : B_{d_n-2} \to B_{d_n-1}$ is injective.
\end{proof}

Proposition \ref{prop:wlpmultinom} shows that there is an infinite family of WLP-algebras which are very close to satisfy the sufficient conditions in Theorem \ref{thm:wlplimit}.

\begin{example} \label{ex:WLPprop410}
Consider the algebras  $A=k[x_1, \ldots, x_5]/(x_1^2,x_2^2, x_3^2, x_4^6, x_5^8)$  and $B=k[x_1, \ldots, x_5]/(x_1^2,x_2^2, x_3^2, x_4^6, x_5^7)$, with $k$ of characteristic 5. Then $t_A=1+1+1+5+7=15$, and $t_B=14$, so $(t_A+1)/2=8$ and $(t_B+1)/2>7$. Neither $A$ nor $B$ satisfy the condition in Theorem \ref{thm:wlplimit}. But since $\binom{8}{1,1,1,5}$ is not divisible by 5, they both have the WLP. 
\end{example}
\begin{remark}
Example \ref{ex:WLPprop410} can be generalized to $d_1=d_2=\dots = d_{n-2}=2$, $d_{n-1}=p+1$, $d_n=p+n-2$, as long as $n-2<p$. 
\end{remark}

We have now arrived at our last sufficient condition of the WLP.

\begin{proposition}\label{prop:wlpx2}
Let $A=k[x_1, \ldots, x_n]/(x_1^{d_1}, \ldots, x_n^{d_n})$ 
and suppose that $t=\sum_{i=1}^n(d_i-1)$ is odd. If $A$ has the WLP, so does $A[x]/(x^2)$. 
\end{proposition}
\begin{proof}
 Let $B=A[x]/(x^2)$, and notice that $B_d=A_d \oplus xA_{d-1}$. Let $s$ be a weak Lefschetz element in $A$. Obviously $sxA_i \subseteq xA_{i+1}$, and the multiplication map $\cdot s: xA_i \to xA_{i+1}$ is injective or surjective exactly when $\cdot s: A_i \to A_{i+1}$ is. Notice that $\cdot s: A_{\frac{t-1}{2}} \to A_{\frac{t+1}{2}}$ is bijective. The map $\cdot s: B_i \to B_{i+1}$ is injective when both $\cdot s: A_i \to A_{i+1}$ and $\cdot s: xA_{i-1} \to xA_{i}$ are so, which is when $i\le (t-1)/2$. In the same way we see that it is surjective when $i \ge (t+1)/2$, and we conclude that $\cdot s: B_i \to B_{i+1}$ has maximal rank for all $i$.
\end{proof}

A reasonable question at this point is if Proposition \ref{prop:wlpx2} can prove the WLP for some algebra, for which it is not already proven by any of the other results in this section. The next example shows that this is the case. 

\begin{example}
Let $k$ be a field of characteristic 5. By Proposition \ref{prop:wlpmultinom}, the algebra $k[x_1, \ldots, x_5]/(x_1^3, x_2^3, x_3^6, x_4^6, x_5^{14})$ has the WLP. It follows from Proposition \ref{prop:wlpx2} that the algebra $k[x_1, \ldots, x_6]/(x_1^3, x_2^3, x_3^6, x_4^6, x_5^{14}, x_6^2)$ also has the WLP. This algebra is not covered by Theorem \ref{thm:wlplimit}. One may try to use Proposition \ref{prop:wlpmultinom} (with the algebra $B$), but this fails because $\binom{15}{1,2,2,5,5}$ is divisible by 5. 
\end{example}

Figure \ref{fig1} and Table \ref{table:wlpdist} is an approach of describing our current understanding of the WLP for mixed degrees. 
  Running the experiments in Table 1 takes almost three hours on a desktop computer. Thus we are far from being able to generate enough data to get a detailed picture.

 \begin{figure}
  \begin{center}
    \begin{tikzpicture}
      \begin{scope}
    \clip \secondcircle;
      \end{scope}
      \begin{scope}
    \clip \firstcircle;
      \end{scope}
      \draw \firstcircle node[text=black,above] {$B$};
      \draw \secondcircle;
       \node[text width=3cm] at (0.4,-1.2) {$C$};
      \draw \thirdcircle;
       \node[text width=3cm] at (0.4+2,-1.2) {$D$};
            \draw \fourthcircle;
   \node[text width=3cm] at (0,0) {$A$}; 
    \end{tikzpicture}
    \end{center}

       \caption{A very rough sketch of the sufficient conditions for the presence of the WLP. The area $A$  represents the space of parameters for which Theorem \ref{thm:wlplimit} does not apply, but where the WLP holds. The area $B$ represents the space of parameters for which we can use Theorem \ref{thm:wlpfamily} to reduce to a case where the WLP has already been detected. More precisely, the parameters in $B$ are of the form $(d_1,\ldots, d_{i-1},d_{i}+bp^a, d_{i+1}, \ldots, d_{j-1}, d_j+bp^a, d_{j+1},\ldots ,d_{n})$, with $a,b \geq 1$ and such that  $k[x_1,\ldots,x_n]/(x_1^{d_1},\ldots,x_n^{d_n})$ has the WLP. The area $C$ represents the space where Proposition \ref{prop:wlpmultinom} can be used, that is, the parameters in $C$ are of the form 
$(d_1,\ldots,d_n)$ and $(d_1,\ldots,d_{n-1},d_n-1)$ such that  $\binom{d_n}{d_1-1,\ldots, d_{n-1} -1}$ is not divisible by the characteristic. Finally, the area $D$ represents the space where Proposition \ref{prop:wlpx2} can be used, that is, the parameters in $D$ are of the form 
$(d_1,\ldots,d_{n-1},2)$ such that  $k[x_1,\ldots,x_{n-1}]/(x_1^{d_1},\ldots,x_{n-1}^{d_{n-1}})$ has the WLP.}

       \label{fig1}
       \end{figure}
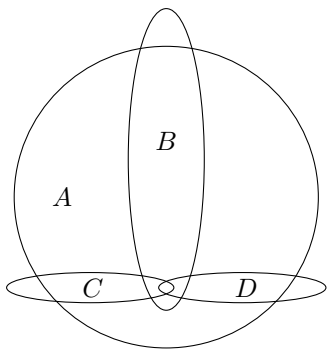
       
\begin{table}[h!]
$$
\begin{array}{cc|ccccc}
     &  p & |A| & |B| &  |C| &  |D| & |A \setminus (B \cup C \cup D)| \\
       \hline
       d_1 = 2& 5 &455 & 68 & 132 & 334 & 43 \\ 
       d_1 \geq 4& 5 &142 & 0 & 68 & 0 & 74 \\ 
      d_1 = 2 & 7 & 821 & 195 & 154 & 568 & 134 \\ 
      d_1 \geq 3 & 7 & 145 & 29 & 36 & 0 & 74 \\  
          d_1 = 2 & 11 & 833 & 550 &  154 & 498 & 86 \\ 
      d_1 \geq 3 & 11 & 318 & 272  & 51 & 0 &  45 \\ 
    d_1 = 2 & 13 & 1374 & 1071 &  250 & 775 & 109 \\ 
          d_1 \geq 3 & 13 & 621 & 540 &  100 & 0 & 81 \\

\end{array}
$$
         \caption{The distribution of the WLP for the set of algebras $\mathbb{Z}/p \mathbb{Z} [x_1,x_2,x_3,x_4,x_5]/(x_1^{d_1}, x_2^{d_2}, x_3^{d_3}, x_4^{d_4}, x_5^{d_5}), d_1  \leq d_2 \leq d_3 \leq d_4 \leq d_5 \leq 25, p \in \{5,7,11,13\}$. The sets $A$, $B$, $C$ and $D$ are as in Figure \ref{fig1}. The last column is the number of WLP algebras which we are not able to detect using Theorem \ref{thm:wlpfamily}, Proposition \ref{prop:wlpmultinom}, and Proposition \ref{prop:wlpx2}.
  Calculations were done with Macaulay2 \cite{M2}.}
    \label{table:wlpdist}
\end{table}

\begin{remark}
Vraciu \cite{vraciu} considers the harder problem of deciding the least degree of a zero-divisor on $k[x_1,\ldots,x_n]/(x_1^{d_1}, \ldots, x_n^{d_n})$, and can draw some conclusion on the WLP for mixed degrees when $k$ is infinite. When $n \geq 5$, the tuple $(d_1,\ldots,d_5,p) = (4,4,4,4,5,3)$ was detected \cite[Proposition 5.6]{vraciu}, and when $n=4$, a large family of WLP algebras is described \cite[Proposition 5.6]{vraciu}.
\end{remark}

\subsection{The Fr\"oberg conjecture in positive characteristic}
The results on the WLP has some applications to the Fr\"oberg conjecture in positive characteristic.
Fr\"oberg's original conjecture is over $\mathbb{C}$.

\begin{conjecture}[Fr\"oberg \cite{froberg}] \label{conj:froberg}
Let $R = \mathbb{C}[x_1,\ldots,x_n]$. Then there exist forms $f_1,\ldots,f_m$ of degrees $d_1,\ldots,d_m$ with $m \geq n+1$ such that
$(R/(f_1,\ldots,f_m))(t) = \left[\prod(1-t^{d_i})/(1-t)^n \right].$
\end{conjecture}

The conjecture has been shown to be true when $n = 2$, see \cite{froberg},  when $n = 3$, see \cite{anick}, and when $m = n + 1$, see \cite{stanley}. Moreover, Hochster and Laksov \cite{hochsterlaksov} has shown that the series is correct up to degree $\min(d_1,\ldots,d_m)+1$.

The conjecture is believed to be true not only over $\mathbb{C}$, but over any field $k$. Fr\"oberg's proof of the two variable case is independent of the underlying field, while Anick's proof in three variables holds in characteristic $p$ only if $k$ is infinite. The Hochster-Laksov result also works in characteristic $p$ under the assumption that $k$ is infinite. In \cite{nicklasson}, it is shown that when $m$ is large enough and the degrees $d_i$ are uniform, then the Hochster-Laksov result holds also when $k$ is finite.

However, Stanley's result in the $m=n+1$ case \emph{does} require a field of characteristic zero. Thus the $m=n+1$ case is the most attractive unproven part of the conjecture in positive characteristic. We will show below that partial results can be derived by using a connection to the WLP.

\begin{theorem}
Fr\"oberg's conjecture over a field $k$ of characteristic $p$ in $n$ variables and $n+1$ forms of degrees $d_1,\ldots,d_{n+1}$ holds true if the algebra $k[x_1,\ldots,x_{n+1}]/(x_1^{d_1}, \ldots, x_{n+1}^{d_{n+1}})$
has the WLP. Especially, the results on the WLP in Theorem \ref{thm:wlplimit}, Theorem \ref{thm:wlpfamily}, Proposition \ref{prop:wlpmultinom}, and Proposition \ref{prop:wlpx2} are positive results on the Fr\"oberg conjecture in $n-1$ variables and $n$ forms of degrees $d_1,\ldots,d_n$.
\end{theorem}
\begin{proof}
 Let $R = k[x_1,\ldots,x_n]/(x_1^{d_1}, \ldots, x_n^{d_n}, (x_1+\ldots + x_n)^{d_{n+1}}).$ By Lemma \ref{lemma:wlp-maxrank}, $R(t) =   \left[\prod(1-t^{d_i})/(1-t)^n \right]$ if and only if
$k[x_1,\ldots,x_{n+1}]/(x_1^{d_1}, \ldots, x_{n+1}^{d_{n+1}})$ has the WLP. Thus if $k[x_1,\ldots,x_{n+1}]/(x_1^{d_1}, \ldots, x_{n+1}^{d_{n+1}})$ has the WLP, the Fr\"oberg conjecture is satisfied for 
$n+1$ forms of degrees $d_1,\ldots,d_{n+1}$ in $k[x_1,\ldots,x_n]$. 
\end{proof}


\begin{remark} \label{rmk:noninjective}
Even if we replace $(x_1+\cdots + x_n)^{d_{n+1}}$ by a general form $f$ of degree $d_{n+1}$, we should not expect that the multiplication by $f$ on the algebra $k[x_1,\ldots,x_n]/(x_1^{d_1}, \ldots, x_n^{d_n})$
has maximal rank in each degree in characteristic $p$. For instance, let $A = k[x_1,\ldots, x_6]/(x_1^2,\ldots,x_6^2)$, with $k$ of characteristic $2$. Then multiplication by any form $f$ of degree two is not injective, since $H(2) = H(4) = 15$, but $f^2 = 0$ in $A$. It follows that in order to attack the Fr\"oberg conjecture in full for $m = n + 1$ by means of specific forms, another choice than powers of the variables has to be used.
\end{remark}

\begin{problem} 

Inspired by Remark \ref{rmk:noninjective}, we would like to address the problem of the characterization of the algebras for which there exists a form $f$ such that $k[x_1,\ldots,x_n]/(x_1^{d_1},\ldots,x_n^{d_n}, f)$ has the Hilbert series conjectured by Fr\"oberg, with $k$ of positive characteristic.
\end{problem}

\begin{ack}
The authors would like to thank Christian Gottlieb for valuable comments on drafts on this manuscript.
\end{ack}


\begin{thebibliography}{99}

\bibitem{anick}
D. J. Anick, Thin Algebras of Embedding Dimension Three, J. Algebra, 100 (1986), 235--259.

\bibitem{brenner} H. Brenner, A. Kaid, A note on the weak Lefschetz property of monomial complete intersections in positive characteristic, Collect. Math. 62 (2011), no. 1, 85--93.

\bibitem{cook} D. Cook II, The Lefschetz properties of monomial complete intersections in positive characteristic, J. Algebra 369 (2012), 42--58.

\bibitem{froberg}
R. Fr\"oberg, An inequality for Hilbert series of graded algebras, Math. Scand. 56 (1985), 117--144.


\bibitem{M2} D. Grayson, M. Stillman,
          Macaulay2, a software system for research
                   in algebraic geometry, Available at {\tt www.math.uiuc.edu/Macaulay2}.
                   
\bibitem{harima} T. Harima, J. Migliore, U. Nagel and J. Watanabe, The Weak and Strong
Lefschetz Properties for Artinian K-Algebras, J. Algebra 262 (2003), 99--126.

\bibitem{hochsterlaksov}
M. Hochster, D. Laksov, The linear Syzygies of generic forms, Comm. Algebra, 15, no. 1-2 (1987), 227--239       
   

\bibitem{kustin}
A. Kustin, A. Vraciu, The weak Lefschetz property for monomial complete intersections in positive
characteristic, Trans. Amer. Math. Soc. 366 (2014) 4571--4601.

\bibitem{li}
J. Li and F. Zanello, Monomial complete intersections, the weak Lefschetz propery and place
partitions, Discrete math. 310 (2010), no. 24, 3558--3570.

\bibitem{lind}
M. Lindsey, A class of Hilbert Series and the strong Lefschetz property, Proc. Amer. Math. Soc. 139 (1) (2011) 79--92.

   
\bibitem{mig_miro-roig}
J. Migliore and R. M. Mir\'o-Roig, Ideals of general forms and the ubiquity of the weak Lefschetz property, J. Pure Appl. Algebra 182 (2003), no. 1, 79--107.

\bibitem{migliore} J. Migliore, R. Miró-Roig, U. Nagel, Monomial ideals, almost complete intersections and the weak Lefschetz property, Trans. Amer. Math. Soc. 363 (2011), 229--257.

\bibitem{atour} J. Migliore, U. Nagel, A tour of the weak and strong Lefschetz properties, J. Commut. Algebra
5, no.3 (2013), 329--358.

\bibitem{nicklasson} L. Nicklasson, On the Hilbert series of ideals generated by generic forms, to appear in Comm. Algebra, DOI 10.1080/00927872.2016.1236931.
                   
\bibitem{reid}
L. Reid, L.G. Roberts, M. Roitman, On complete intersections and their Hilbert functions, Canad. Math. Bull. Vol. 34 no. 4 (1991), 525--535.

\bibitem{stanley}
R. Stanley, Weyl groups, the hard Lefschetz theorem and the Sperner property, Siam J. Alg. Disc. Math. 1 (1980), 168--184.

\bibitem{vraciu} A. Vraciu, On the degrees of relations on $x_1^{d_1}, \ldots, x_n^{d_n}, (x_1+\cdots+x_n)^{d_{n+1}}$ in positive characteristic, J. Algebra, 423 (2015), 916--949.

\end{thebibliography}
\end{document}